\documentclass[12pt]{amsart}
\usepackage{amsmath,amssymb, amscd, graphicx}

\makeatletter
\@addtoreset{equation}{section}
\makeatother
\usepackage{enumerate}
\usepackage{color}

\def\ii{{\sqrt{-1}}}

\def\hX{{\hat{X}}}

\def\abl{{{\widetilde{w}}}}
\def\Abl{{{w}}}

\def\cK{{\mathcal{K}}}

\def\CC{{\mathbb C}}
\def\ZZ{{\mathbb Z}}
\def\NN{{\mathbb N}}

\def\RR{{\mathbb R}}

\def\cO{{\mathcal{O}}}

\def\JJ{{\mathcal J}}

\def\hGamma{{{\Gamma_\tau}}}

\def\fB{{\mathcal{B}}}

\newtheorem{theorem}{Theorem}[section]
\newtheorem{definition}[theorem]{Definition}

\newtheorem{proposition}[theorem]{Proposition}

\newtheorem{remark}[theorem]{Remark}
\newtheorem{lemma}[theorem]{Lemma}

\def\book#1{\rm{#1}, }
\def\paper#1{\textit{#1}, }
\def\jour#1{\rm{#1}, }
\def\yr#1{({\rm{#1}) }}
\def\vol#1{\textbf{#1}}
\def\pages#1{\rm{#1}}

\def\by#1{{\rm{#1}, }}

\begin{document}

\title{The Riemann constant for a non-symmetric 
Weierstrass semigroup}

\author{Jiryo Komeda, Shigeki Matsutani and Emma Previato}


\maketitle

\begin{abstract}
The zero divisor of the
theta function of a compact Riemann surface $X$ of genus $g$
is the canonical theta divisor of Pic${}^{(g-1)}$ up to translation
by the Riemann constant $\Delta$ for a base point $P$ of $X$. 
The complement of the Weierstrass gaps at the base point $P$ 
given as a numerical semigroup plays an important role, which
is called the Weierstrass semigroup.
It is classically known that the Riemann constant $\Delta$ 
is  a half period $\frac{1}{2}\Gamma_\tau$
for the Jacobi variety $\JJ(X)=\CC^g/\Gamma_\tau$ of $X$
if and only if 
the Weierstrass semigroup at $P$ is symmetric.
In this article, we analyze the non-symmetric case.
Using a semi-canonical divisor $D_0$,
we show a relation between the Riemann constant $\Delta$
and a half period $\frac{1}{2}\Gamma_\tau$ of the non-symmetric case. 
We also identify the semi-canonical divisor $D_0$ for 
trigonal curves, and
remark on an algebraic expression for the Jacobi inversion problem
using the relation.
\end{abstract}


{\centerline{\textbf{2000 MSC: 
14H55, 
14H50, 
14K25, 
14H40 
}}}

\section{Introduction}

The Riemann constant is an important invariant that is associated to a pointed
curve $(X,P)$, whose Abel map is normalized at $P$.
In this article we work over the complex numbers, assume that $X$ is a
compact Riemann surface of genus $g>1$ (which we call simply a curve),
 and use standard convention, in particular
identifying divisors under addition with line bundles under multiplication,
the Jacobian of $X$ with the complex torus $\JJ(X):=\CC^g/\Gamma_\tau,$ 
where
$\Gamma_\tau$ is the period lattice associated with the choice of a 
 standard homology basis
$\langle \alpha_i, \beta_i\rangle_{1\le i \le g}$
on $X$. We call the corresponding basis of
normalized holomorphic forms $\omega_i, {1\le i \le g}$,
we denote by $\Abl$   the Abel map,    by   $\kappa$
the natural projection $\CC^g \to \JJ(X)$, and by $\cK_X$
the canonical divisor of
$X$. 

Despite the choices involved, there is a valuable uniqueness to the Riemann
constant: it is the vector $\Delta\in
 \CC^g$ 
such that, for a given base point $P$,
$\displaystyle{\theta 
\left(\int_{g P}^{P_1+\cdots +P_{g}}\omega-\Delta\right)=0}$ if and
  only if $P_1+\cdots P_{g}-P$ is a positive divisor \cite[Th. 1.1]{fay};
 moreover, 
$2\Delta +\Abl(\cK_X-(g-1)P)=0$ modulo $\Gamma_\tau$
 \cite{fay,Lew}. These facts connect the Riemann constant with the set of
 semi-canonical divisors, which in turn correspond to theta characteristics
in $\frac{1}{2}\mathbb{Z}^{2g}$.



The complement 
of the Weierstrass gaps at $P$ is called the Weierstrass semigroup
if it is a numerical semigroup. We consider the pointed curve $(X,P)$
associated with a Weierstrass semigroup.

When the curve $X$ has  the property that 
a canonical divisor $\cK_X =(2g-2)P$, or
equivalently, the number $(2g-1)$ is (the last) Weierstrass gap,
the Weierstrass semigroup at $P$ is symmetric, i.e., an integer $n$
belongs to it if and only if $\ell_g-n$ does not, where
$1<\ell_1<\ell_2<\cdots <\ell_g$ are the Weierstrass gaps at $P$
\cite{stoehr}. This is also the case if and only if the Riemann constant is a
half period, 
$\Delta\in\frac{1}{2}\Gamma_\tau$:
 the parity and addition theorems for the theta function 
become simpler than in general \cite{fay}. It is only in this case that, to
 the best of our knowledge, the Riemann constant has been written explicitly;
 in the hyperelliptic case, it is $\kappa\Delta=\tau [\frac{1}{2},
 ...,\frac{1}{2}]+[\frac{g}{2},\ \frac{g-1}{2} ...,\frac{1}{2}]$ \cite[IIIa
 (5.4)]{mumford} and in the trigonal cyclic case of the Picard curve, it is
$\kappa\Delta=
\tau [0\ \frac{1}{2}\ 0]+[0\ \frac{1}{2}\ 0]$ \cite[Prop. I-2]{shiga}.

The Weierstrass semigroup at $P$ may fail to be symmetric, although
the problem of finding all numerical semigroups that can be realized on a
pointed curve is still open (cf. \cite{pinkham}). As far as we know,
the results in \cite{MK,KMP} about the Jacobi inversion problem are the only
ones given so far for the non-symmetric case. We were able to 
generalize results of Mumford's \cite{mumford} on 
hyperelliptic functions, in particular for the 
 semigroups  $\langle3,4,5\rangle$,
and $\langle3,7,8\rangle$.

In this article, we state the algebraic-transcendental correspondence
for the Riemann constant on a general pointed curve and its consequences for
the Jacobi inversion problem. 
In the recent monograph \cite{FZ}, the authors analyze the Riemann constant
as well as we do, for $Z_n$ curves, which have a total
ramification at $P$; their goal is to derive (explicitly in some cases) an
analog of  the classical Thomae
formula, which expresses algebraically the branch points of the curve in terms
of theta characteristics and thetanulls. Our motivation instead is
to use the 
the {\lq\lq}shifted Riemann constant{\rq\rq} $\Delta_s$, 
cf. Definition \ref{def:SRC},
and {\lq\lq}shifted Abel map{\rq\rq} $\Abl_s$, cf. Definition \ref{def:SAM},
to make explicit the correspondence between the group structure 
 of the Jacobian and linear equivalence of divisors on the curve, a 
crucial issue when evaluating the sigma function (associated to the theta
function) on symmetric products of the curve, cf. Section \ref{last}.
More precisely
we introduce a divisor $\fB$ related to a semi-canonical divisor $D_0$
in Lemma \ref{lemm:a1}, and using $\fB$, we define
the shifted Riemann constant $\Delta_s$ and
and shifted Abel map $\Abl_s$ to connect the shifted
Riemann constant $\Delta_s$ with a half period $\frac{1}{2}\Gamma_\tau$.
The main results are given in 
Theorems \ref{thm:01}, \ref{thm:0} and \ref{thm:1}.
The case when $P$ is a trigonal point is made explicit in Proposition
\ref{trigonal}. As in Remark \ref{rmk:JIF} and (\ref{eq:theta}),
we show that
they play crucial roles in the Jacobi inversion problem.



After setting up notation and standard facts in Section \ref{notation},
 we explain our
construction of the non-symmetric examples in Section 3 and interpret them in
the classical setting of theta characteristics in Section \ref{last}.
\bigskip

\noindent
{\bf{Acknowledgments: }}
One of the authors (S.M.) thanks Atsushi Nakayashiki for 
pointing out this problem for 
\cite{KMP,MK}, and Yoshihiro \^Onishi for critical discussions.

\bigskip
\section{Notation and review}\label{notation}
\subsection{Abel map and theta functions}

Let us consider a compact Riemann surface $X$ of genus $g$
and its Jacobian $\JJ(X):={\CC}^g/\Gamma_\tau$ where
$\Gamma_\tau:={\ZZ}^g +\tau {\ZZ}^g$.
Let $\hX$ be an Abelian covering of $X$ ($\varpi: \hX \to X$).
Since the covering space $\hX$ is constructed by a quotient space
of path space (contour of integral),
we consider an embedding $X$ into $\hX$ by a map $\iota: X \to \hX$
such that $\varpi\circ \iota = id$.
For a point $P\in \hX $, we define the Abel map $\Abl$
and $\abl$
$$
\Abl: S^k \hX \to \CC^g, \quad \Abl(P_1, \cdots, P_k) =
 \sum_{i=1}^k \Abl(P_i) =
 \sum_{i=1}^k \int^{P_i}_P \omega, \quad
\abl:=\kappa \Abl \iota: S^k X \to \JJ(X),
$$
where $S^k \hX$ and $S^k X$ are
$k$-symmetric products of $\hX$ and $X$ respectively.
The Abel theorem shows $\kappa \Abl = \abl \varpi$.
We fix $\iota$ and $\varpi P$ is also denoted by $P$.

Later we will fix the point $P \in X$ as a marked point $P$ of $(X,P)$.

The map $\abl$ embeds $X$ into $\JJ(X)$
and generalizes to a map from the space of divisors of $X$
into $\JJ(X)$ as $\abl(\sum_i n_i P_i):=\sum_in_i\abl(P_i)$, $P_i\in X$,
$n_i\in\ZZ$. Similarly we will define $\Abl(\iota D)$ for 
a divisor $D$ of $X$. (In Introduction, we omitted $\iota$ and
wrote $\Abl(D)$ rather than $\Abl(\iota D)$.)

The Riemann theta function, analytic in both variables $z$ and $\tau$, is
defined by:
$$
\theta(z,\tau ) =
\sum_{n \in \ZZ^{g}} 
\exp\left( 2\pi\ii ({}^t n z + \frac{1}{2}{}^t n \tau n)\right)
.
$$ 
 The zero-divisor of $\theta$ modulo $\Gamma_\tau$ 
is denoted by  $\Theta:=\kappa \mathrm{div}(\theta) \subset \JJ(X)$.

The theta function with characteristics $\delta', \delta''\in\RR^{g}$
is defined as:
\begin{equation}
\theta \left[\begin{matrix}\delta'\\ \delta''\end{matrix}
\right] (z, \tau )
   =
   \sum_{n \in \ZZ^g} \exp \big[\pi \sqrt{-1}\big\{
    \ ^t\negthinspace (n+\delta')
      \tau(n+\delta')
   + 2\ {}^t\negthinspace (n+\delta')
      (z+\delta'')\big\}\big].
\label{eq:2.1}
\end{equation}

If $\delta = (\delta',\delta'')\in 
  \{0, \frac{1}{2}\}^{2g}$, then $\theta
\left[\delta\right]\left(z,\tau\right):=\theta
\left[^{\delta'}_{\delta''}\right]\left(z,\tau\right)$ has definite
parity in $z$, $\theta \left[\delta\right]
\left(-z,\tau\right)=e(\delta) \theta \left[\delta\right]
\left(z,\tau\right)$, where $e(\delta):=e^{4\pi \imath{\delta'}^t
\delta''}$. There are $2^{2g}$ different characteristics of definite
parity.


\subsection{Numerical  semigroups}

A numerical semigroup $H=\langle M\rangle$ generated by a set $M$,
has gap sequence $L(H):=\NN_0\setminus H$ of $H$
and a finite number 
$g(H)$, called ``genus'', of elements in $L(H)$.
For example,
$$
L(\langle 3,4,5\rangle)=\{1, 2\}, \ \ \ 
L(\langle 3,7,8\rangle)=\{1, 2, 4, 5\}. \ \ \ 
$$
For a gap sequence $L:=\{\ell_0 < \ell_1 < \cdots < \ell_{g-1} \}$ of 
genus $g$, let $M(L)$ be the minimal set of generators for the 
semigroup $H(L)$ and 
\begin{equation}
\alpha(L) :=\{\alpha_0(L), \alpha_1(L), \ldots, \alpha_{g-1}(L)\}
\label{eq:alphaL}
\end{equation}
where $\alpha_i := \ell_i - i -1$.  When an $\alpha_i$
is repeated $j>1$ times we write $\alpha_i^j$ in $\alpha(L)$.
We let: $\lambda_i(L):=\alpha_{g-i}(L)+1$ and associate to $L$ the Young
diagram 
$\Lambda(L) =(\lambda_1(L), \ldots, \lambda_g(L))$.

We say that $H$ is symmetric when it has a property that
an integer $n$ belongs to it if and only if $\ell_g-n$ does not.
 Therefore, $H$ is symmetric if and only if
$\Lambda(L(M))$ is self-dual; in this case, the Young diagram 
is equal to  its transpose. 
It is known that $H$ is symmetric if and only if
$2g(H)-1$ is a gap of $H$.

We let $a_{min}(L)$ be the smallest positive integer of $M(L)$.
We call a semigroup $H$ an $a_{min}(L)$-semigroup; for example,
$\langle3,4,5\rangle$ and $\langle3,7,8\rangle$ are $3$-semigroups.

\medskip
For a 
 curve $X$ of genus $g$
 and a point $P \in X$,  the semigroup 
$$ 
 H(X,P):= \{n \in \NN_0\ |\ \mbox{there exists } f \in k(X)
                          \mbox{ such that } (f)_\infty = n P\ \},
$$
 called the Weierstrass semigroup of the point $P$, 
is a numerical semigroup and $g=g(H(X,P))$.
If $L(H(X,P)):=\NN_0\setminus H(X,P)$ differs from the set 
$\{1, 2, \cdots, g\}$, 
we say that $P$ is a Weierstrass point of $X$.
 We also say that the pointed curve  $(X,P)$ is symmetric (non-symmetric) 
if such is the Weierstrass semigroup at $P$.

\section{Canonical divisor of a pointed curve}

In this section we give a lemma on the canonical divisors 
of a pointed curve $(X,P)$ with 
 Weierstrass semigroup $H(X,P)$, including non-symmetric ones.
The fact is classical, cf., e.g., \cite[Th. A-3]{shiga},
\cite[Definition 3.9, p.163]{mumford0}.

\begin{lemma}\label{lemm:a1}
There is a divisor $\fB$ of degree 
$d_0:=\mathrm{deg}(\fB)$ such that
\begin{gather*}
\cK_X=2D_0
= 2(g-1+d_0)P-2\fB, 
\end{gather*}
where we indicate linear equivalence by an equal sign.
\end{lemma}

\begin{proof}
By the surjectivity of the Abel map,
there exists a divisor $D_0$ such that
$\cK_X=2 D_0$. ($D_0$ is called the semi-canonical divisor.)
Since the degree of $\cK_X$ is $2g-2$,
there is a divisor $\fB$ such that
$D_0=(g-1+d_0)P-\fB$.
\end{proof}

\begin{remark}\label{rmk:W-curve}
{\rm{
\begin{enumerate}

\item As we show later, the divisor $\fB$ plays much more important
roles than the semi-canonical divisor $D_0$ in this article,
though $\fB$ is defined from $D_0$. 

\item
If  $H(X,P)$ is symmetric, $\fB$ can be taken to be the zero divisor
 and $d_0=0$, therefore
Lemma \ref{lemm:a1} includes the symmetric
Weierstrass semigroup at $P$ as a special case.

\item
A pointed curve hyperelliptic at $P$ (namely, $H(X,P)$ contains the number
 2), is symmetric.
If the Weierstrass semigroup of the curve is generated by
two numbers $\langle r,s\rangle$, the semigroup is  symmetric. 
For example, a trigonal  non-singular plane curve 
$y^3 =x^r+1$ is a  $\langle 3, r\rangle$-type symmetric
curve (at $P=\infty$). 
Since a trigonal curve with total ramification at $P$ has semigroup
generated by either two or three elements, this remark and Proposition
\ref{trigonal} cover all trigonal cases with total ramification.

\item
The semi-canonical divisor $D_0$ i.e., $2D_0=\cK_X$ is sometimes called 
{\it{theta characteristics}} \cite[Definition 3.9, p.163]{mumford0}.
Instead of this terminology, 
we refer to the vector $\delta$ in (\ref{eq:2.1}) as the
{\it{theta characteristics}}.

\end{enumerate}
}}
\end{remark}

\begin{proposition}\label{trigonal}
For a pointed curve $(X,P)$ whose Weierstrass semigroup at $P$ is of type
$\langle 3,2r+s,2s+r\rangle$ for natural numbers $(r,s)\neq 1$, $r>s$
we have the following results:

\begin{enumerate}
\item 
$(X,P)$ is not symmetric.

\item 
The genus of $X$ is $g=r+s-1$;  and

\item 
the divisor $\fB$ in Lemma \ref{lemm:a1} can be written explicitly as
$\fB=B_{s+1}+\cdots+B_{s+r}$, with
$$
\cK_X=(2g-2)P+B_1+\cdots+B_{s}-s P,$$
where $B_1, \ldots, B_{r+s}$ are ramification points 
corresponding to the branch points of a singular curve,
$$
f_0(x,y)= y^3-(x-b_1)\cdots(x-b_s)\cdot
 (x-b_{s+1})^2\cdots(x-b_{s+r})^2=0,
$$
with $P=\infty$,
$$
B_1+\cdots+B_{s}+2(B_{s+1}+\cdots+B_{s+r})-(s+2r)P\sim 0,
$$
but
$$
B_{s+1}+\cdots+B_{s+r}- rP\not\sim 0.
$$
\end{enumerate}
\end{proposition}
\begin{proof}
As indicated in \cite{MK, KMP}, where the result was proven for
the cases
$\langle3,4,5\rangle$ and $\langle3,7,8\rangle$  corresponding to
$(r=1,s=2)$ and $(r=2,s=3)$ respectively, the result holds in general.
More precisely, by normalizing the singular curve, we also have \cite{MK},
$$
f_1(x,y)= w^3-(x-b_1)^2\cdots(x-b_s)^2\cdot
 (x-b_{s+1})\cdots(x-b_{s+r})=0.
$$
In other words, we have 
the commutative ring $R=\CC[x,y,w]/(h_1, h_2, h_3)$ of the curve $(X,P)$,
$R = \cO_X(*P)$, 
where
\begin{gather*}
\begin{split}
h_1(x,y) &= w^2 - (x-b_1)\cdots(x-b_s)y =0,\\
h_2(x,y) &= y^2 - (x-b_{s+1})\cdots(x-b_{s+r})w =0,\\
h_3(x,y) &= wy - (x-b_{1})\cdots(x-b_{s+r}) =0,\\
\end{split}
\end{gather*}
By simple computations, we have
$$
\left(\frac{1}{w}dx\right)=(B_1+\cdots+B_{s})+(2g-2-s)P.
$$
\end{proof}

\begin{remark}\label{rmk:trigonal}
{\rm{
We  give some comments on 
the pointed curve $(X,P)$ of type $\langle 3,2r+s,2s+r\rangle$
treated in Proposition \ref{trigonal}.
The curve $\langle3,7,8\rangle$ and the related cyclic singular
curves are also studied in \cite[p. 83]{FZ} (without
consideration of the Weierstrass  semigroup),
for the generalizations of Thomae's formula.
In this article, we consider these singular curves 
for the Jacobi inversion formulae, so we focus on the affine ring
 $\mathrm{Spec} R$ by  normalizing these curves.
This allows us to use linear-equivalence relations, such as
$3B_i \sim 0$ and
\begin{gather*}
\begin{split}
B_1+\cdots+B_{s}+B_{s+1}+\cdots+B_{s+r}-(s+r)P
&\sim 
-(B_{s+1}+\cdots+B_{s+r}-rP)\\
&\sim 
2(B_{s+1}+\cdots+B_{s+r}-rP).\\
\end{split}
\end{gather*}
By letting $\fB_1:=B_1+\cdots+B_{r+s}$, $\fB_1 - (s+r)P \sim 2(\fB-rP)$.
Further we have
$$
\left(\frac{1}{wy}dx\right)=-\fB_1 +(2g-2+r+s)P.
$$
We define $R^B:=\{ f \in R \ | \ \exists \ell, \ 
\mbox{such that}\ (f)- \fB_1 +\ell P >0\}$, so that 
$y$ and $w$ belong to $R^B$.
We choose a graded (by order of pole) basis of 
$R^B = \oplus_{i=0} \CC f_i$ as a $\CC$-vector space.
Then the holomorphic one-forms are explicitly written as
$$
\frac{f_i}{wy}dx \quad i = 0, \cdots, g-1.
$$
By Riemann-Roch theorem, the degree of $f_{g-2}$ at $P$ is
$(2g-2)+(r+s)$.

We display the examples of $R$ and $R^B$:

\begin{gather*}
{\tiny{
\centerline{
\vbox{
	\baselineskip =10pt
	\tabskip = 1em
	\halign{&\hfil#\hfil \cr
        \multispan7 \hfil Table 1: Examples of $R$ \hfil \cr
	\noalign{\smallskip}
	\noalign{\hrule height0.8pt}
	\noalign{\smallskip}
$(r,s) $ \strut\vrule  &
$g \backslash$wt &\strut\vrule &
0 &1 & 2 & 3 & 4 & 5 & 6 & 7 & 8 & 9 & 10 & 11 & 12 & 13 & 14 & 15 & 16
& 17 & 18 & \cr
\noalign{\smallskip}
\noalign{\hrule height0.3pt}
\noalign{\smallskip}
$(1,3)$ \strut\vrule &
$3$ &\strut\vrule &
 1& - & - & $x$ & - & $w$ & $x^2$& $y$ & $xw$ & $x^3$
& $x y$ & $x^2 w$ & $w y$ & $x^2y$ & $x^3w$ & $x wy$
& $x^3y$ & $w^2 y$ & $x^2 wy$ \cr
$(2,3)$ \strut\vrule &
$4$ &\strut\vrule &
 1& - & - & $x$ & - & - & $x^2$& $w$ & $y$ & $x^3$
& $x w$ & $x y$ & $x^4$ & $x^2w$ & $x^2 y$ & $wy$ &
 $x^3w$ & $x^3 y$ & $x wy$ \cr
$(1,5)$ \strut\vrule &
$5$ &\strut\vrule &
 1& - & - & $x$ & - & - & $x^2$& $w$ & - & $x^3$ & $xw$
& $y$ & $x^4$ & $x^2w$ & $x y$ & $x^5$
& $x^3w$ & $x^2y$ & $wy$
\cr
$(2,4)$ \strut\vrule &
$5$ &\strut\vrule &
 1& - & - & $x$ & - & - & $x^2$& - & $w$ & $x^3$ & $y$
& $xw$ & $x^4$ & $xy$ & $x^2 w$ &  $x^5$ &
 $x^2y$ & $x^3w$ & $wy$ \cr
$(3,4)$ \strut\vrule &
$6$ &\strut\vrule &
 1& - & - & $x$ & - & - & $x^2$& - & -& $x^3$ & $w$
& $y$ & $x^4$ & $xw$ & $x y$ &  $x^5$ &
 $x^2w$ & $x^2y$ & $x^6$ \cr
\noalign{\smallskip}
	\noalign{\hrule height0.8pt}
}
}
}
}}
\end{gather*}

\begin{gather*}
{\tiny{
\centerline{
\vbox{
	\baselineskip =10pt
	\tabskip = 1em
	\halign{&\hfil#\hfil \cr
        \multispan7 \hfil Table 2: Examples of  $R^B$ \hfil \cr
	\noalign{\smallskip}
	\noalign{\hrule height0.8pt}
	\noalign{\smallskip}
$(r,s) $ \strut\vrule  &
$g \backslash$wt &\strut\vrule &
0 &1 & 2 & 3 & 4 & 5 & 6 & 7 & 8 & 9 & 10 & 11 & 12 & 13 & 14 & 15 & 16
& 17 & 18 & \cr
\noalign{\smallskip}
\noalign{\hrule height0.3pt}
\noalign{\smallskip}
$(1,3)$ \strut\vrule &
$3$ &\strut\vrule &
 -& - & - & - & - & $w$ & -& $y$ & $xw$ & -
& $x y$ & $x^2 w$ & $w y$ & $x^2y$ & $x^3w$ & $x wy$
& $x^3y$ & $w^2 y$ & $x^2 wy$ \cr
$(2,3)$ \strut\vrule &
$4$ &\strut\vrule &
 -& - & - & - & - & - & -& $w$ & $y$ & -
& $x w$ & $x y$ & - & $x^2w$ & $x^2 y$ & $wy$ &
 $x^3w$ & $x^3 y$ & $x wy$ \cr
$(1,5)$ \strut\vrule &
$5$ &\strut\vrule &
 -& - & - & - & - & - & -& $w$ & - & - & $xw$
& $y$ & - & $x^2w$ & $x y$ & -
& $x^3w$ & $x^2y$ & $wy$
\cr
$(2,4)$ \strut\vrule &
$5$ &\strut\vrule &
 -& - & - & - & - & - & - & - & $w$ & - & $y$
& $xw$ & - & $xy$ & $x^2 w$ &  - &
 $x^2y$ & $x^3w$ & $wy$ \cr
$(3,4)$ \strut\vrule &
$6$ &\strut\vrule &
 -& - & - & - & - & - & - & - & -& - & $w$
& $y$ & - & $xw$ & $x y$ &  - &
 $x^2w$ & $x^2y$ & - \cr
\noalign{\smallskip}
	\noalign{\hrule height0.8pt}
}
}
}
}}
\end{gather*}

}}
\end{remark}.

\section{Riemann constant}\label{last}

 From Theorem 7 in \cite{Lew}, we have
\begin{proposition} \label{prop:lew}
$$
\abl({\mathcal{S}}^{g-1}X) + \kappa\Delta = \Theta \quad
\mbox{modulo} \quad \Gamma_\tau
$$
by letting the Riemann constant denoted by $\Delta \in \CC^g$.
\end{proposition}

It implies that, up to translation
by the vector $\Delta$, the zero-divisor $\Theta$ of $\theta$
is the canonical theta divisor of Pic${}^{(g-1)}$.

Theorem 11 in \cite{Lew} says:
\begin{proposition} \label{prop:Lew2}
An effective divisor $D$ whose degree is $2g-2$
satisfies $\abl(D-(2g-2)P) +2 \kappa \Delta =0$  modulo $\Gamma_\tau$
if and only if $D$ is the divisor of the holomorphic one form,
i.e., $\abl(\cK_X-(2g-2)P) +2 \kappa \Delta=0$ modulo $\Gamma_\tau$.
\end{proposition}

As mentioned in the Introduction,  the following result is classical (e.g,
  \cite[Appendix, Cor. 2]{shiga}):
\begin{lemma}
The Riemann vector $\Delta$ belongs to $\frac{1}{2}\Gamma_\tau$ if
and only if $\cK_X=(2g-2)P$.
\end{lemma}

\begin{definition} \label{def:SRC}
We define the {\lq\lq}shifted Riemann constant{\rq\rq}
 by a translation:
$$ 
\Delta_s:=\Delta -\Abl(\iota \fB) \in \CC^g.
$$
\end{definition}

\begin{theorem}\label{thm:01}
$\Delta_s$ belongs to $\frac{1}{2}\hGamma$.
\end{theorem}
\begin{proof}
>From Proposition \ref{prop:Lew2} and Lemma \ref{lemm:a1},
$$
\abl(\cK_X-(2g-2)P)+2\kappa
\Delta=\abl(\cK_X-(2g-2)P+2\fB) +2\kappa
\Delta_s=0
\quad\mbox{modulo}\quad \Gamma_\tau.
$$
On the other hand,  the first term vanishes, because
$$\abl(\cK_X-(2g-2)P+2\fB) = 
\abl(0)=0
\quad\mbox{modulo}\quad \Gamma_\tau.
$$
Hence $2\Delta_s = 0$ modulo $\hGamma$ and the statement is proved. 
We note that this property does not depend on the choice of the 
identification $\iota$.
\end{proof}

\begin{remark}{\rm{
Theorem \ref{thm:01} gives a correspondence between the Riemann constant
of a general pointed curve and  a half-period point of the Jacobian.

There are $2^{2g}$ elements of 
$(\frac{1}{2}\Gamma_\tau)/\Gamma_\tau$,
which is bijective to a set $\Sigma$ whose elements $D$ 
are defined by $2D=\cK_X$
\cite[p.163]{mumford0}, \cite[Appendix]{shiga}. 
Therefore, the choice of $D_0$ in 
Lemma \ref{lemm:a1} typically is not explicit, rather, 
the Lemma is an existence statement.
For a  curve as in Proposition \ref{trigonal},
$\fB$ is naturally determined. 

Our construction is very similar to  
the relation  between the half-period and $D-D_0$ of $D, D_0 \in \Sigma$
\cite[Th. A-4]{shiga}. However, 
as in Remark \ref{rmk:W-curve} (1),
we use different divisors; especially,
$\fB$ does not belong to $\Sigma$ in general.
}}
\end{remark}

\begin{definition} \label{def:SAM}
The  {\lq\lq}shifted Abel map{\rq\rq} is defined by
$$
\Abl_s(P_1, \ldots, P_k) = \Abl(P_1, \ldots, P_k) + \Abl(\iota\fB),
$$
and $\abl_s:=\kappa \Abl_s$.
In other words it is given by translating a divisor $D$,
$\Abl_s(\iota D)=\Abl(\iota (D+\fB))$ and  
and $\abl_s(D)=\abl(D+\fB)$. 
\end{definition}

\begin{theorem}\label{thm:0}
$\abl_s(P_1, \ldots, P_{g-1}) + \kappa \Delta_s$ is 
in the theta divisor $\Theta$ modulo $\Gamma_\tau$; or more 
explicitly,
$$
\abl_s({\mathcal{S}}^{g-1}X) + \kappa \Delta_s = \Theta
\quad \mbox{modulo}\quad \Gamma_\tau
$$
using the shifted Abel map.
\end{theorem}

It also implies that, up to translation
by the vector $\Delta_s$, the zero-divisor $\Theta$ of $\theta$
is the canonical theta divisor of Pic${}^{(g-1)}$ using
the shifted Abel map.

\begin{proof}
$\abl_s(P_1, \ldots, P_{g-1}) + \kappa \Delta_s =
\abl(P_1, \ldots, P_{g-1}) +\kappa  \Delta_s + \abl(\fB)
=\abl(P_1, \ldots, P_{g-1}) + \kappa \Delta$.
Proposition \ref{prop:lew} means that the left hand side of
the formula is $\Theta$.
\end{proof}

In the paper \cite{KMP}, we call $\Delta_s$ itself {\it{the Riemann constant}}
since it plays such role under the shifted Abel map.


\begin{theorem}\label{thm:1}
There exists a theta characteristic
$\displaystyle{\delta=
\left[\begin{matrix}\delta'\\ \delta''\end{matrix} \right]
  \in \left\{0, \frac{1}{2}\right\}^{2g}}$
so that
$$
\theta 
\left[\begin{matrix}\delta'\\ \delta''\end{matrix} \right]
 (\Abl_s(P_1, \ldots, P_{g-1})) = 0
$$
for every $P_i \in \hX$, i.e., 
for $\Theta_s:=\kappa \mathrm{div}(\theta[\delta])$, 
$\Theta_s = \abl_s(S^{g-1} X)$. 
\end{theorem}

\begin{remark}{\rm{
As in \cite{KMP}, our goal is to extend the classical knowledge of Jacobi's
inversion problem, where
the theta function with characteristics
$\delta \in \{0, \frac{1}{2}\}^{2g}$
plays a central role. Ultimately,
Jacobi's inversion connects the meromorphic and the
Abelian functions of a curve.
Theorem \ref{thm:1} enables us to use the properties of
the theta function with theta characteristics 
$\delta \in \{0, \frac{1}{2}\}^{2g}$
for non-symmetric
pointed curves as well.

More precisely, our previous work was concerned with
 the ``sigma function'', a
generalization of Weierstrass' elliptic
sigma function; the higher-genus case sigma is, up to an exponential
multiplicative factor  quadratic in the exponent, 
the theta function with theta characteristics
$\delta \in \{0, \frac{1}{2}\}^{2g}$.

In fact, in  \cite{MK, KMP}, we defined the sigma function
and solved the Jacobi inversion problem, but  
although we used the above 
Theorems \ref{thm:01}, \ref{thm:0} and \ref{thm:1} implicitly,
we defined the shifted Riemann constant by checking
  \textit{ad hoc} calculations.
 In this article, we streamline the theory, which is based on 
simple group-theoretic properties of the Jacobian, 
and have a geometric meaning given in  
the following Remark.
}}
\end{remark}

\begin{remark} \label{rmk:JIF}
{\rm{
Assume that $\fB_1$ is an effective divisor such that
$d_1:=\deg(\fB_1)<2d_0$ and $\fB_1 - d_1 P \sim 2 \fB-2d_0 P$
as in Proposition \ref{trigonal} and Remark \ref{rmk:trigonal}.
Let us define $R = \cO_X(*P)$, the ring of meromorphic functions with 
 pole at most at $P$, and 
$R^B:=\{ f \in R \ | \ \exists \ell, \ 
\mbox{such that}\ (f)-\fB_1 +\ell P >0\}$. 
Thus, $R^B$ is a natural
commutative ring that keeps track of the Weierstrass semigroup of a pointed
curve. 

We choose a graded basis of 
$R^B = \oplus_{i=0} \CC f_i$ as a $\CC$-vector space
arising from the elements of $H(X,P)$.
Let $n$ be a positive integer  and 
$P_1, \ldots, P_n$ be in $X\backslash P$.
We define the {\it{Frobenius-Stickelberger (FS) determinant}} by
$$
\psi_{n}(P_1, P_2, \ldots, P_n) := 
\left|\begin{matrix}
 f_0(P_1) & f_1(P_1)  &\cdots  & f_{n-1}(P_1) \\
 f_0(P_2) & f_1(P_2)  &\cdots  & f_{n-1}(P_2) \\
\vdots & \vdots & \ddots & \vdots\\
 f_0(P_n) & f_1(P_n)  &\cdots  & f_{n-1}(P_n) \\
\end{matrix}\right|
$$
and $\mu$-functions by
$$
\mu_n(P): = 
\mu_n(P; P_1,  \ldots, P_n): = 
\lim_{P_i' \to P_i}\frac{1}{\psi_{n}(P_1' , \ldots, P_n' )}
\psi_{n+1}(P_1' , \ldots, P_n' , P),
$$
where the $P_i^\prime$ are generic,
the limit is taken (irrespective of the order) for each $i$.
Here we also have $\mu_{n, k}(P_1, \ldots, P_n)$ by
$$
\mu_n(P)
 = f_n(P) + 
\sum_{k=0}^{n-1} (-1)^{n-k}\mu_{n, k}(P_1, \ldots, P_n) f_k(P),
$$
with the convention $\mu_{n, n}(P_1, \ldots, P_n) \equiv 1$.
The divisor $(\mu_n(P))$ shows that there are points $Q_i \in X$ 
such that 
$$\sum_{i=1}^{n}  P_i 
+ \sum_{i=1}^{N(n)-n-2d_0} Q_i - N(n) P + \fB_1
\sim 0
$$
where $N(n)\in H(X,P)$ is the order of $\mu_n$.
Noting the definition of $\fB_1$
This implies that the addition structure of the shifted Abelian map satisfies:
$$
\left(\sum_{i=1}^{n}  P_i+\fB\right)
 - (n +d_0)P \sim 
- \left(\sum_{i=1}^{N(n)-n-d_1} Q_i +\fB\right)
 - (N(n)-n-d_1 + d_0)P.
$$
For the $n=g-1$ case, by assuming the fact
that $N(g-1)=2g-2 + d_1$ as in Remark \ref{rmk:trigonal},
we have
$$
\left( \sum_{i=1}^{g-1}  P_i +\fB\right)
- (g-1 +d_0)P \sim 
- \left(\sum_{i=1}^{g-1} Q_i +\fB\right)
 - (g-1+d_0)P.
$$
This corresponds to a symmetry 
  of $\Theta_s$ in Theorem \ref{thm:1} under the minus operation
on the Jacobian as in \cite[p.166]{mumford0}, which 
 makes the shifted Abel map natural, i.e.,
\begin{equation}
\Theta_s = - \Theta_s, \qquad 
\abl_s(S^{g-1}X) =-\abl_s(S^{g-1}X).
\label{eq:theta}
\end{equation}
Further by the explicit Jacobi inversion formulae, we connect the theta 
function with half-integer theta characteristics and
$\mu_{n,k}$, cf. \cite{KMP,MK}.
}}
\end{remark}

\begin{remark}
{\rm{
As mentioned in the Introduction, we introduced 
the {\lq\lq}shifted Riemann constant{\rq\rq} in Definition \ref{def:SRC}
and {\lq\lq}shifted Abel map{\rq\rq} Definition \ref{def:SAM}
which enable us to handle theta function with 
theta characteristics in Theorem \ref{thm:1}.
Our results hold for general pointed curves with Weierstrass semigroup
generated by a lowest integer $n$,
not necessarily $Z_n$ curves as in \cite{FZ}. Our main motivation 
is an explicit expression of linear equivalence of divisors in terms of
the addition structure of the Jacobian as a complex torus.
}}
\end{remark}

\bigskip
\bigskip

\noindent
Jiryo Komeda:\\
Department of Mathematics,\\
Center for Basic Education and Integrated Learning,\\
Kanagawa Institute of Technology,\\
Atsugi, 243-0292, JAPAN.\\
e-mail: komeda@gen.kanagawa-it.ac.jp\\
\\
Shigeki Matsutani:\\
Industrial Mathematics,\\
National Institute of Technology, Sasebo College,\\
1-1 OkiShin-machi, Sasebo, Nagasaki, 857-1193, JAPAN\\
e-mail: smatsu@sasebo.ac.jp\\
\\
Emma Previato:\\
Department of Mathematics and Statistics,\\
Boston University,\\
Boston, MA 02215-2411, U.S.A.\\
e-mail:ep@bu.edu\\

\end{document}